\pdfoutput=1

\documentclass{article}%
\usepackage{amsmath}%
\usepackage{amsfonts}%
\usepackage{amssymb}%
\usepackage{graphicx}
\newtheorem{theorem}{Theorem}

\newtheorem{corollary}[theorem]{Corollary}

\newtheorem{proposition}[theorem]{Proposition}
\newtheorem{remark}[theorem]{Remark}

\newenvironment{proof}[1][Proof]{\textbf{#1.} }{\ \rule{0.5em}{0.5em}}

\begin{document}

\title{Approximating with Gaussians}
\author{Craig Calcaterra and Axel Boldt\\Metropolitan State University\\{\small craig.calcaterra@metrostate.edu}}
\maketitle

\begin{abstract}
Linear combinations of translations of a single Gaussian, $e^{-x^{2}}$, are
shown to be dense in $L^{2}\left(  \mathbb{R}\right)  $. Two algorithms for
determining the coefficients for the approximations are given, using
orthogonal Hermite functions and least squares. Taking the Fourier transform
of this result shows low-frequency trigonometric series are dense in $L^{2}$
with Gaussian weight function.



\end{abstract}

\section{Linear combinations of Gaussians with a single variance are dense in
$L^{2}$}

$L^{2}\left(  \mathbb{R}\right)  $ denotes the space of square integrable
functions $f:\mathbb{R}\rightarrow\mathbb{R}$ with norm $\left\|  f\right\|
_{2}:=\sqrt{\int_{\mathbb{R}}\left|  f\left(  x\right)  \right|  ^{2}dx}$. We
use $f\underset{\epsilon}{\approx}g$ to mean $\left\|  f-g\right\|
_{2}<\epsilon$. The following result was announced in \cite{CalcGaussians}.

\begin{theorem}
\label{ThmBumps}For any $f\in L^{2}\left(  \mathbb{R}\right)  $ and any
$\epsilon>0$ there exists $t>0$ and $N\in\mathbb{N}$ and $a_{n}\in\mathbb{R}$
such that%
\[
f\underset{\epsilon}{\approx}\ \overset{N}{\underset{n=0}{%
{\textstyle\sum}
}}a_{n}e^{-\left(  x-nt\right)  ^{2}}\text{.}%
\]
\end{theorem}

\begin{proof}
Since the span of the Hermite functions is dense in $L^{2}\left(
\mathbb{R}\right)  $ we have for some $N$%
\begin{equation}
f\underset{\epsilon/2}{\approx}\ \overset{N}{\underset{n=0}{%
{\textstyle\sum}
}}b_{n}\frac{d^{n}}{dx^{n}}\left(  e^{-x^{2}}\right)  \text{.}%
\label{PfMainHermite}%
\end{equation}
Now use finite backward differences to approximate the derivatives. We have
for some small $t>0$%
\begin{align}
& \overset{N}{\underset{n=0}{%
{\textstyle\sum}
}}b_{n}\frac{d^{n}}{dx^{n}}\left(  e^{-x^{2}}\right) \nonumber\\
& \underset{\epsilon/2}{\approx}b_{0}e^{-x^{2}}+b_{1}\tfrac{1}{t}\left[
e^{-x^{2}}-e^{-\left(  x-t\right)  ^{2}}\right]  +b_{2}\tfrac{1}{t^{2}}\left[
e^{-x^{2}}-2e^{-\left(  x-t\right)  ^{2}}+e^{-\left(  x-2t\right)  ^{2}%
}\right] \nonumber\\
& +b_{3}\tfrac{1}{t^{3}}\left[  e^{-x^{2}}-3e^{-\left(  x-t\right)  ^{2}%
}+3e^{-\left(  x-2t\right)  ^{2}}-e^{-\left(  x-3t\right)  ^{2}}\right]
+\cdots\nonumber\\
& =\overset{N}{\underset{n=0}{%
{\textstyle\sum}
}}b_{n}\frac{1}{t^{n}}\overset{n}{\underset{k=0}{%
{\textstyle\sum}
}}\left(  -1\right)  ^{k}\binom{n}{k}e^{-\left(  x-kt\right)  ^{2}}%
\text{.}\label{PfMainCoeffs}%
\end{align}
\end{proof}

This result may be surprising; it promises we can approximate to any degree of
accuracy a function such as the following characteristic function of an
interval%
\[
\chi_{\left[  -11,-10\right]  }\left(  x\right)  :=\left\{
\begin{array}
[c]{l}%
1\\
0
\end{array}
\right.
\begin{array}
[c]{l}%
\text{for }x\in\left[  -10,-11\right] \\
\text{otherwise}%
\end{array}
\]
with support far from the means of the Gaussians $e^{-\left(  x-nt\right)
^{2}}$ which are located in $\left[  0,\infty\right)  $ at the points $x=nt$.
The graphs of these functions $e^{-\left(  x-nt\right)  ^{2}}$ are extremely
simple geometrically, being Gaussians with the same variance. We only use the
right translates, and they all shrink precipitously (exponentially) away from
their means.

\noindent%
\[%
{\parbox[b]{2.5261in}{\begin{center}
\includegraphics[
natheight=9.364200in,
natwidth=13.749600in,
height=1.7244in,
width=2.5261in
]%
{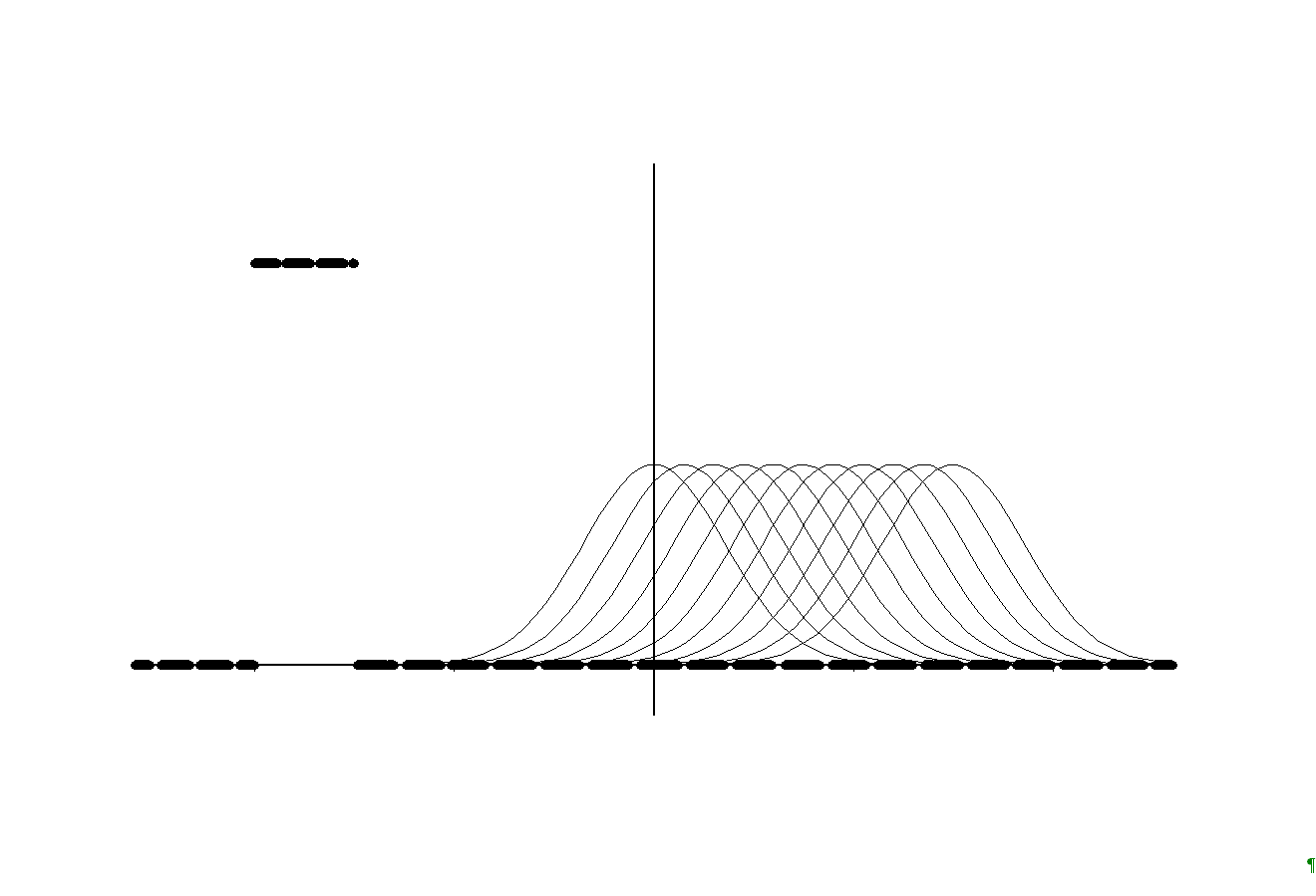}%
\\
$\sum a_n e^{-\left(  x-nt\right)  ^2}\approx$ characteristic function?
\end{center}}}%
\]

\bigskip\bigskip

\textit{Surely there is a gap in this sketchy little proof?}

No. We will, however, flesh out the details in section \ref{SecAppAlgo}. The
coefficients $a_{n}$ are explicitly calculated and the $L^{2}$ convergence
carefully justified. But these details are elementary. We include them in the
interest of appealing to a broader audience.

\textit{Then is this merely another pathological curiosity from analysis? We
probably need impractically large values of }$N$\textit{\ to approximate any
interesting functions.}

No, $N$ need only be as large as the Hermite expansion demands. Certainly this
particular approach depends on the convergence of the Hermite expansion, and
for many applications Hermite series converge slower than other Fourier
approximations--after all, Hermite series converge on all of $\mathbb{R}$
while, e.g., trigonometric series focus on a bounded interval. Hermite
expansions do have powerful convergence properties, though. For example,
Hermite series converge uniformly on finite compact subsets whenever $f$ is
twice continuously differentiable (i.e., $C^{2}$) and $O\left(  e^{-cx^{2}%
}\right)  $ for some $c>1$ as $x\rightarrow\infty$. Alternately if $f$ has
finitely many discontinuities but is still $C^{2}$ elsewhere and $O\left(
e^{-cx^{2}}\right)  $ the expansion again converges uniformly on any closed
interval which avoids the discontinuities \cite{Stone}, \cite{Szego}:. If $f$
is smooth and properly bounded, the Hermite series converges faster than
algebraically \cite{Gottlieb}.

\textit{Then is the method unstable?}

Yes, there are two serious drawbacks to using Theorem \ref{ThmBumps}.

\noindent1. Numerical differentiation is inherently unstable. Fortunately we
are estimating the derivatives of Gaussians, which are as smooth and bounded
as we could hope, and so we have good control with an explicit error formula.
It is true, though, that dividing by $t^{n}$ for small $t$ and large $n$ will
eventually lead to huge coefficients $a_{n}$ and round-off error. There are
quite a few general techniques available in the literature for combatting
round-off error in numerical differentiation. We review the well-known
$n$-point difference formulas for derivatives in section \ref{SecImpConv}.

\noindent2. The surprising approximation is only possible because it is weaker
than the typical convergence of a series in the mean. Unfortunately%
\[
f\left(  x\right)  \neq\overset{\infty}{\underset{n=0}{%
{\textstyle\sum}
}}a_{n}e^{-\left(  x-nt\right)  ^{2}}%
\]
Theorem \ref{ThmBumps} requires recalculating all the $a_{n}$ each time $N$ is
increased. Further, the $a_{n}$ are not unique. The least squares best choice
of $a_{n}$ are calculated in section \ref{SecLeastSqrs}, but this approach
gives an ill-conditioned matrix. A different formula for the $a_{n}$ is given
in Theorem \ref{AlgBump} which is more computationally efficient.

Despite these drawbacks the result is worthy of note because of the new and
unexpected opportunities which arise using an approximation method with such
simple functions. In this vein, section \ref{SecLowFTrig} details an
interesting corollary of Theorem \ref{ThmBumps}: apply the Fourier transform
to see that low-frequency trigonometric series are dense in $L^{2}\left(
\mathbb{R}\right)  $ with Gaussian weight function.

\section{Calculating the coefficients with orthogonal
functions\label{SecAppAlgo}}

In this section Theorem \ref{AlgBump} gives an explicit formula for the
coefficients $a_{n}$ of Theorem \ref{ThmBumps}. Let's review the details of
the Hermite-inspired expansion%
\[
f\left(  x\right)  =\overset{\infty}{\underset{n=0}{%
{\textstyle\sum}
}}b_{n}\frac{d^{n}}{dx^{n}}\left(  e^{-x^{2}}\right)
\]
claimed in the proof. The formula for these coefficients is%
\[
b_{n}:=\tfrac{1}{n!2^{n}\sqrt{\pi}}\underset{\mathbb{R}}{%
{\textstyle\int}
}f\left(  x\right)  e^{x^{2}}\frac{d^{n}}{dx^{n}}\left(  e^{-x^{2}}\right)
dx\text{.}%
\]
Be warned this is not precisely the standard Hermite expansion, but a simple
adaptation to our particular requirements. Let's check this formula for the
$b_{n}$ using the techniques of orthogonal functions.

Remember the following properties of the Hermite polynomials $H_{n}$
(\cite{Szego}, e.g.). Define $H_{n}\left(  x\right)  :=\left(  -1\right)
^{n}e^{x^{2}}\frac{d^{n}}{dx^{n}}e^{-x^{2}}$. The set of Hermite functions%
\[
\left\{  h_{n}\left(  x\right)  :=\dfrac{1}{\sqrt{n!2^{n}\sqrt{\pi}}}%
H_{n}\left(  x\right)  e^{-x^{2}/2}:n\in\mathbb{N}\right\}
\]
is a well-known basis of $L^{2}\left(  \mathbb{R}\right)  $ and is orthonormal
since%
\begin{equation}
\underset{\mathbb{R}}{%
{\textstyle\int}
}H_{m}\left(  x\right)  H_{n}\left(  x\right)  e^{-x^{2}}dx=n!2^{n}\sqrt{\pi
}\delta_{m,n}\text{.}\label{ExHermite2}%
\end{equation}
This means given any $g\in L^{2}\left(  \mathbb{R}\right)  $ it is possible to
write%
\begin{equation}
g\left(  x\right)  =\overset{\infty}{\underset{n=0}{%
{\textstyle\sum}
}}c_{n}\tfrac{1}{\sqrt{n!2^{n}\sqrt{\pi}}}H_{n}\left(  x\right)  e^{-x^{2}%
/2}\label{ExHermite10}%
\end{equation}
$($equality in the $L^{2}$ sense$)$ where%
\[
c_{n}:=\tfrac{1}{\sqrt{n!2^{n}\sqrt{\pi}}}\underset{\mathbb{R}}{%
{\textstyle\int}
}g\left(  x\right)  H_{n}\left(  x\right)  e^{-x^{2}/2}dx\in\mathbb{R}\text{.}%
\]
The necessity of this formula for $c_{n}$ can easily be checked by multiplying
both sides of $\left(  \ref{ExHermite10}\right)  $ by $H_{n}\left(  x\right)
e^{-x^{2}/2}$, integrating and applying $\left(  \ref{ExHermite2}\right)  $.
However, we want%
\[
f\left(  x\right)  =\overset{\infty}{\underset{n=0}{%
{\textstyle\sum}
}}b_{n}\frac{d^{n}}{dx^{n}}e^{-x^{2}}%
\]
so apply this process to $g\left(  x\right)  =f\left(  x\right)  e^{x^{2}/2}$.
But $f\left(  x\right)  e^{x^{2}/2}$ may not be $L^{2}$ integrable. If it is
not, we must truncate it: $f\left(  x\right)  e^{x^{2}/2}\chi_{\left[
-M,M\right]  }\left(  x\right)  $ is $L^{2}$ for any $M<\infty$ and
$f\cdot\chi_{\left[  -M,M\right]  }\underset{\epsilon/3}{\approx}f$ for a
sufficiently large choice of $M$. Now we get new $c_{n}$ as follows%
\begin{align*}
f\left(  x\right)  e^{x^{2}/2}\chi_{\left[  -M,M\right]  }\left(  x\right)
&  =\overset{\infty}{\underset{n=0}{%
{\textstyle\sum}
}}c_{n}\tfrac{1}{\sqrt{n!2^{n}\sqrt{\pi}}}H_{n}\left(  x\right)  e^{-x^{2}%
/2}\text{\qquad so}\\
f\left(  x\right)  \chi_{\left[  -M,M\right]  }\left(  x\right)   &
=\overset{\infty}{\underset{n=0}{%
{\textstyle\sum}
}}c_{n}\tfrac{\left(  -1\right)  ^{n}}{\sqrt{n!2^{n}\sqrt{\pi}}}\left(
-1\right)  ^{n}H_{n}\left(  x\right)  e^{-x^{2}}=\overset{\infty}%
{\underset{n=0}{%
{\textstyle\sum}
}}b_{n}\frac{d^{n}}{dx^{n}}e^{-x^{2}}%
\end{align*}
where%
\begin{align*}
c_{n}  & =\tfrac{1}{\sqrt{n!2^{n}\sqrt{\pi}}}\underset{\mathbb{R}}{%
{\textstyle\int}
}f\left(  x\right)  e^{x^{2}/2}\chi_{\left[  -M,M\right]  }\left(  x\right)
H_{n}\left(  x\right)  e^{-x^{2}/2}\left(  x\right)  dx\\
& =\tfrac{1}{\sqrt{n!2^{n}\sqrt{\pi}}}\underset{\mathbb{R}}{%
{\textstyle\int}
}f\left(  x\right)  \chi_{\left[  -M,M\right]  }\left(  x\right)  H_{n}\left(
x\right)  dx
\end{align*}
so we must have%
\begin{equation}
b_{n}=c_{n}\tfrac{\left(  -1\right)  ^{n}}{\sqrt{n!2^{n}\sqrt{\pi}}}=\tfrac
{1}{n!2^{n}\sqrt{\pi}}\underset{\mathbb{R}}{%
{\textstyle\int}
}f\left(  x\right)  \chi_{\left[  -M,M\right]  }\left(  x\right)  e^{x^{2}%
}\frac{d^{n}}{dx^{n}}e^{-x^{2}}dx\text{.}\label{Line b_n}%
\end{equation}

Now the second step of the proof of Theorem \ref{ThmBumps} claims that the
Gaussian's derivatives may be approximated by divided backward differences%
\[
\frac{d^{n}}{dx^{n}}e^{-x^{2}}\approx\frac{1}{t^{n}}\overset{n}{\underset
{k=0}{%
{\textstyle\sum}
}}\left(  -1\right)  ^{k}\binom{n}{k}e^{-\left(  x-kt\right)  ^{2}}%
\]
in the $L^{2}\left(  \mathbb{R}\right)  $ norm. We'll use the ``big oh''
notation: for a real function $\Psi$ the statement `` $\Psi\left(  t\right)
=O\left(  t\right)  $ as $t\rightarrow0$ '' means there exist $K>0$ and
$\delta>0$ such that $\left|  \Psi\left(  t\right)  \right|  <K\left|
t\right|  $ for $0<\left|  t\right|  <\delta$.

\begin{proposition}
\label{propLpDivDiff}For each $n\in\mathbb{N}$ and $p\in\left(  0,\infty
\right)  $%
\[
\left(  \underset{\mathbb{R}}{\int}\left|  \frac{d^{n}}{dx^{n}}e^{-x^{2}%
}-\frac{1}{t^{n}}%
{\textstyle\sum_{k=0}^{n}}
\left(  -1\right)  ^{k}\binom{n}{k}e^{-\left(  x-kt\right)  ^{2}}\right|
^{p}dx\right)  ^{1/p}=O\left(  t\right)  \text{.}%
\]
\end{proposition}

\begin{proof}
In Appendix \ref{SecImpConv} the pointwise formula is derived:%
\[
\frac{d^{n}}{dx^{n}}g\left(  x\right)  =\frac{1}{t^{n}}%
{\textstyle\sum_{k=0}^{n}}
\left(  -1\right)  ^{k}\binom{n}{k}g\left(  x-kt\right)  -\dfrac{t}{\left(
n+1\right)  !}\overset{n}{\underset{k=0}{%
{\textstyle\sum}
}}\left(  -1\right)  ^{k}\binom{n}{k}k^{n+1}g^{\left(  n+1\right)  }\left(
\xi_{k}\right)
\]
where all of the $\xi_{k}$ are between $x$ and $x+nt$. Therefore the
proposition holds with $g\left(  x\right)  =e^{-x^{2}}$ since $g^{\left(
n+1\right)  }\left(  \xi_{k}\right)  $ is integrable for each $k$. This is not
perfectly obvious because we don't have explicit formulae for the $\xi_{k}$.
But the tails of $g^{\left(  n+1\right)  }$ vanish exponentially, the
continuity of $g^{\left(  n+1\right)  }$ guarantees a finite maximum on the
bounded interval between the tails, and $\left|  \xi_{k}-x\right|  <k\left|
t\right|  $.
\end{proof}

Continuing the derivation of the coefficients $a_{n}$ we now have for
sufficiently small $t\neq0$%
\begin{equation}
f\underset{\epsilon}{\approx}\ \overset{N}{\underset{n=0}{%
{\textstyle\sum}
}}b_{n}\frac{1}{t^{n}}\overset{n}{\underset{k=0}{%
{\textstyle\sum}
}}\left(  -1\right)  ^{k}\binom{n}{k}e^{-\left(  x-kt\right)  ^{2}}%
=\overset{N}{\underset{k=0}{%
{\textstyle\sum}
}}\left[  \overset{N}{\underset{n=k}{%
{\textstyle\sum}
}}b_{n}\frac{\left(  -1\right)  ^{k}}{t^{n}}\binom{n}{k}\right]  e^{-\left(
x-kt\right)  ^{2}}\label{Line f approxi}%
\end{equation}
In the last equality we just switched the order of summation (see
\cite{Knuth}, section 2.4 for an overview of such tricks). Combining $\left(
\ref{Line b_n}\right)  $ and $\left(  \ref{Line f approxi}\right)  $ we have

\begin{theorem}
\label{AlgBump}For any $f\in L^{2}\left(  \mathbb{R}\right)  $ and any
$\epsilon>0$ there exist $N\in\mathbb{N}$ and $t_{0}>0$ such that for any
$t\neq0$ with $\left|  t\right|  <t_{0}$%
\[
f\underset{\epsilon}{\approx}\ \overset{N}{\underset{n=0}{%
{\textstyle\sum}
}}a_{n}e^{-\left(  x-nt\right)  ^{2}}%
\]
for some choice of $a_{n}\in\mathbb{R}$ dependent on $N$ and $t$.

\noindent If $f\left(  x\right)  e^{x^{2}/2}$ is integrable, then one choice
of coefficients is%
\[
a_{n}=\frac{\left(  -1\right)  ^{n}}{n!\sqrt{\pi}}\overset{N}{\underset{k=n}{%
{\textstyle\sum}
}}\tfrac{1}{\left(  k-n\right)  !\left(  2t\right)  ^{k}}\underset{\mathbb{R}%
}{%
{\textstyle\int}
}f\left(  x\right)  e^{x^{2}}\frac{d^{k}}{dx^{k}}\left(  e^{-x^{2}}\right)
dx\text{.}%
\]
If $f\left(  x\right)  e^{x^{2}/2}$ is not integrable, replace $f$ in the
above formula with $f\cdot\chi_{\left[  -M,M\right]  }$ where $M$ is chosen
large enough that $\left\|  f-f\cdot\chi_{\left[  -M,M\right]  }\right\|
_{2}<\epsilon$.
\end{theorem}

\begin{remark}
\label{RemUnifG}The approximation in Theorem \ref{AlgBump} also holds on
$C\left[  a,b\right]  $ with the uniform norm since the Hermite expansion is
uniformly convergent on $C^{2}\left[  a,b\right]  $ (see \cite{Stone},
\cite{Szego}) and the finite difference formula's error term from Appendix
\ref{SecImpConv} converges to 0 uniformly as $t\rightarrow0^{+}$. The
Stone-Weierstrass Theorem does not apply in this situation because linear
combinations of Gaussians with a single variance do not form an algebra.
\end{remark}

\begin{remark}
\label{RemDense}As a consequence of Theorem \ref{AlgBump} for any $\epsilon>0$
the closed linear span of $\left\{  e^{-\left(  x-s\right)  ^{2}}:s\in\left[
0,\epsilon\right)  \right\}  $ is $L^{2}\left(  \mathbb{R}\right)  $. It is
even sufficient to replace $\left[  0,\epsilon\right)  $ with $\left\{
\frac{i}{2^{j}}:i,j\in\mathbb{N}\right\}  \cap\left[  0,\epsilon\right)  $.
\end{remark}

Let's explore some concrete examples in applying Theorem \ref{AlgBump}. Choose
an interesting function with discontinuities and some support negative:%
\[
f\left(  x\right)  :=\left(  x-1\right)  ^{2}\chi_{\left[  -1,2\right]
}\left(  x\right)  :=\left\{
\begin{array}
[c]{l}%
\left(  x-1\right)  ^{2}\\
0
\end{array}
\right.
\begin{array}
[c]{l}%
\text{for }x\in\left[  -1,2\right] \\
\text{otherwise}%
\end{array}
\]
and observe graphically:

\noindent$%
{\parbox[b]{1.7253in}{\begin{center}
\includegraphics[
natheight=6.103900in,
natwidth=8.864300in,
height=1.1917in,
width=1.7253in
]%
{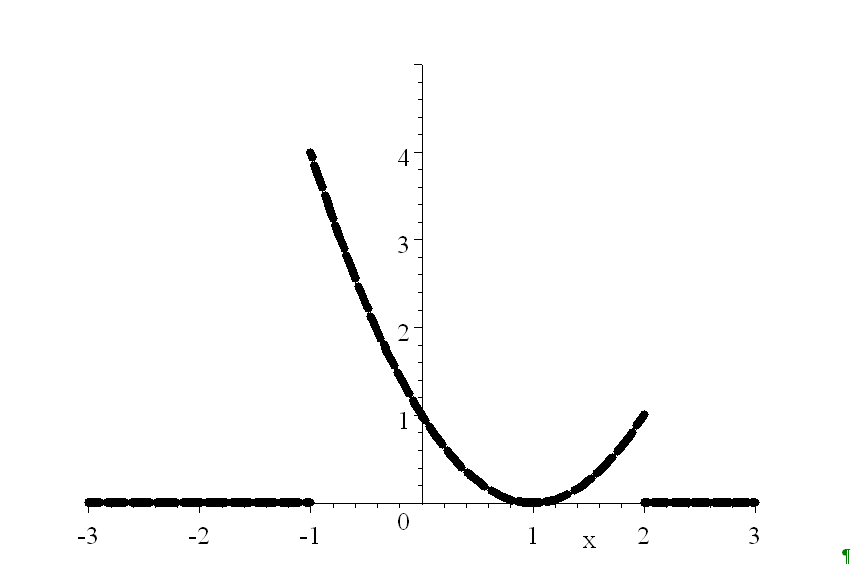}%
\\
$f\left(  x\right)  :=\left(  x-1\right)  ^2\chi_{\left[  -1,2\right]}  \left(
x\right)  $
\end{center}}}%
{\parbox[b]{1.6094in}{\begin{center}
\includegraphics[
natheight=5.917100in,
natwidth=8.583300in,
height=1.1122in,
width=1.6094in
]%
{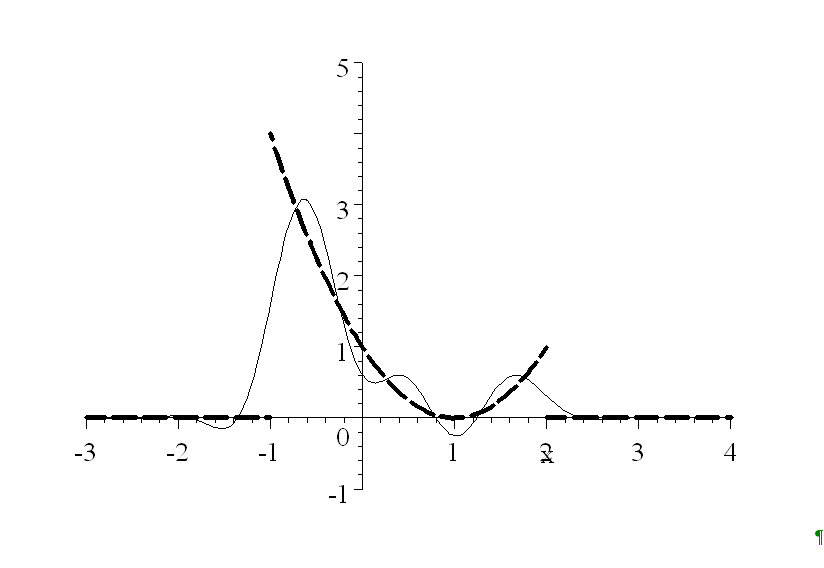}%
\\
Hermite series $N=20$
\end{center}}}%
{\parbox[b]{1.6025in}{\begin{center}
\includegraphics[
natheight=5.822800in,
natwidth=8.437100in,
height=1.1087in,
width=1.6025in
]%
{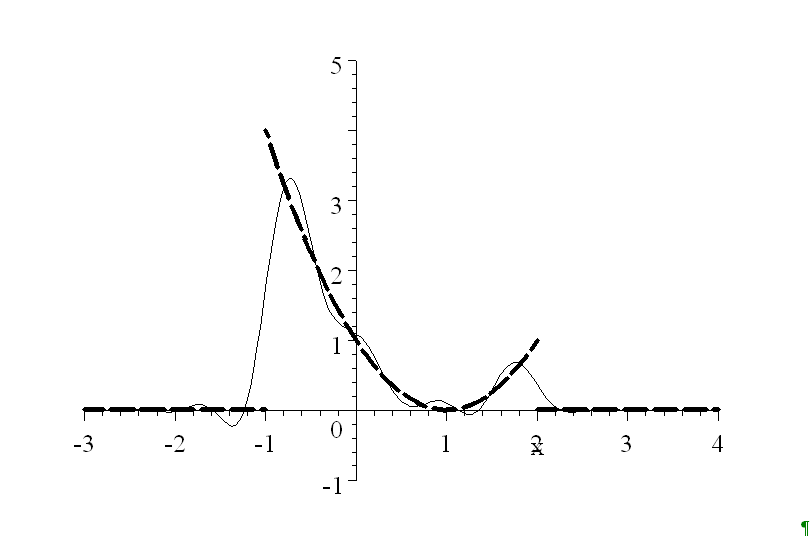}%
\\
Hermite $N=40$
\end{center}}}%
$%
\[%
{\parbox[b]{1.6155in}{\begin{center}
\includegraphics[
natheight=6.229200in,
natwidth=9.052000in,
height=1.1147in,
width=1.6155in
]%
{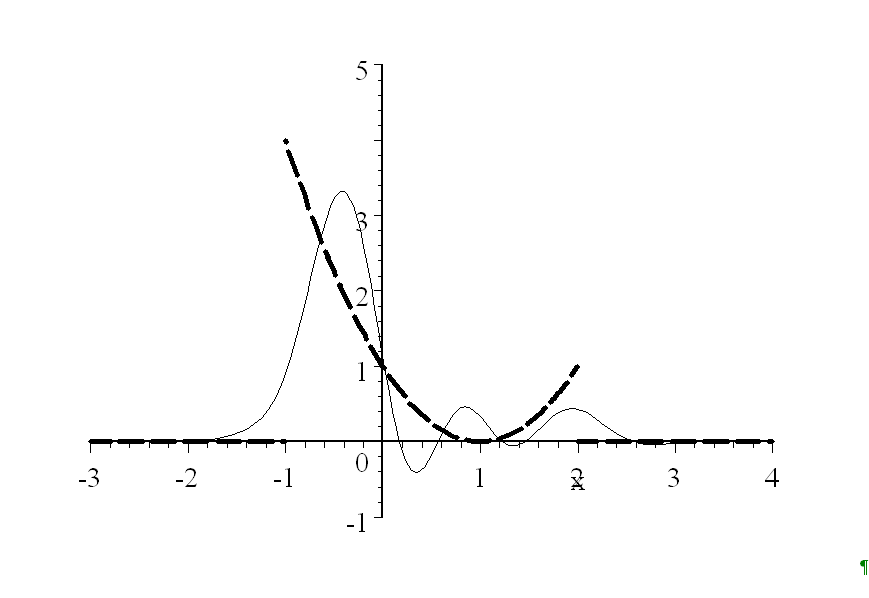}%
\\%
\begin{tabular}
[c]{c}%
Theorem \ref{AlgBump}\\
$N=20$, $t=.05$%
\end{tabular}
\end{center}}}%
{\parbox[b]{1.6172in}{\begin{center}
\includegraphics[
natheight=5.989700in,
natwidth=8.698300in,
height=1.1156in,
width=1.6172in
]%
{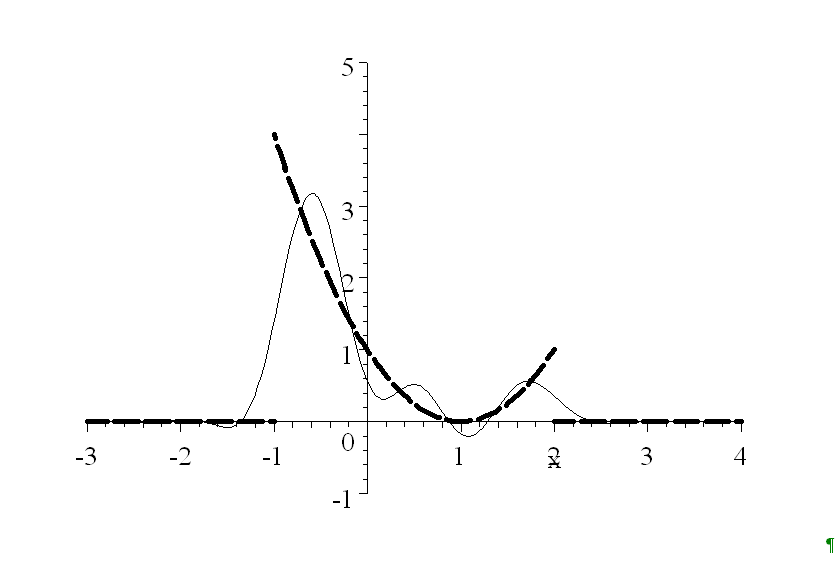}%
\\
{}%
\begin{tabular}
[c]{c}%
Theorem \ref{AlgBump}\\
$N=20$, $t=.01$%
\end{tabular}
\end{center}}}%
{\parbox[b]{1.6163in}{\begin{center}
\includegraphics[
natheight=6.124600in,
natwidth=8.895500in,
height=1.1147in,
width=1.6163in
]%
{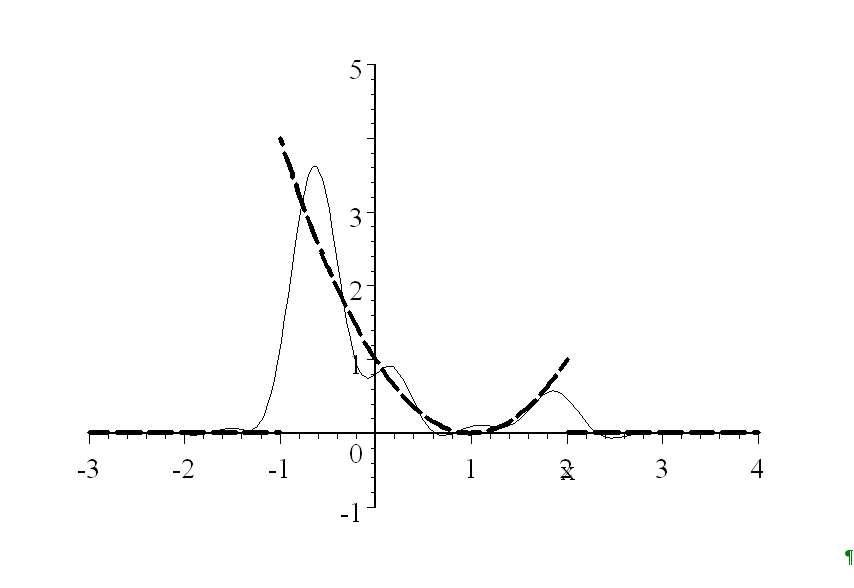}%
\\
{}%
\begin{tabular}
[c]{c}%
Theorem \ref{AlgBump}\\
$N=40$, $t=.01$%
\end{tabular}
\end{center}}}%
\]

\bigskip\bigskip

The Hermite approximation is slowed by discontinuities, but does converge. The
next choice of $f$ is continuous but not smooth.

\noindent%
\[%
{\parbox[b]{1.6414in}{\begin{center}
\includegraphics[
natheight=5.906700in,
natwidth=8.614400in,
height=1.1277in,
width=1.6414in
]%
{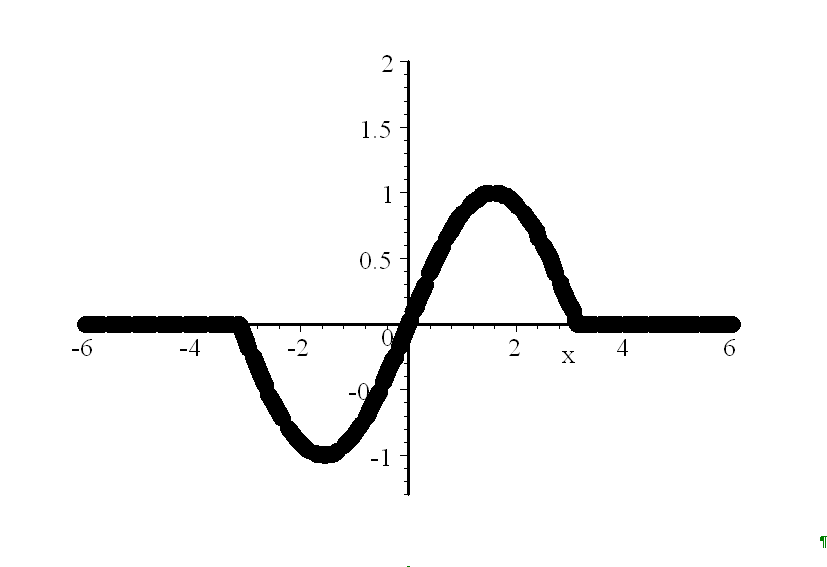}%
\\
$f\left(  x\right)  :=\left(  \sin x\right)  \chi_{\left[  -\pi,\pi\right]}
\left(  x\right)  $
\end{center}}}
{\parbox[b]{1.6622in}{\begin{center}
\includegraphics[
natheight=6.416900in,
natwidth=9.458500in,
height=1.1312in,
width=1.6622in
]%
{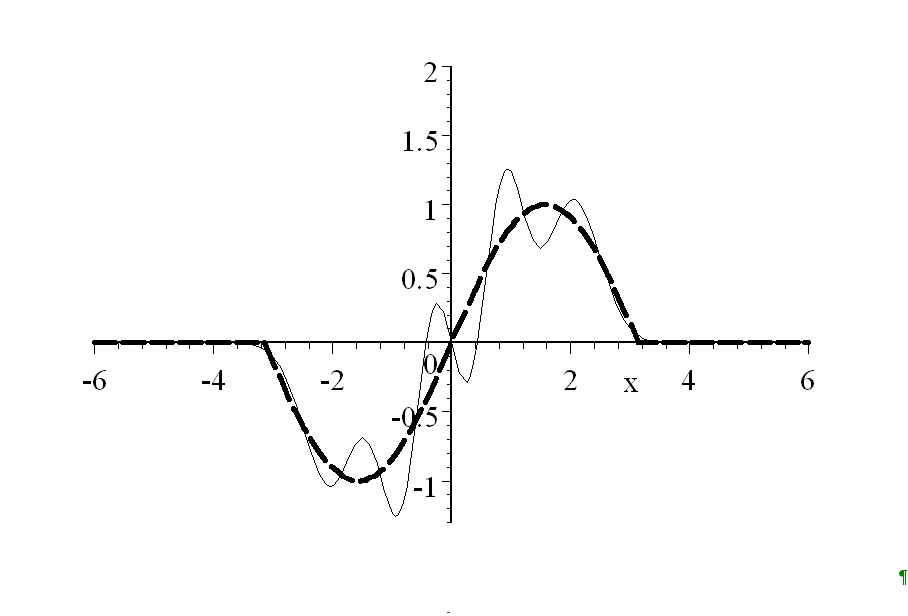}%
\\%
\begin{tabular}
[c]{c}%
Hermite expansion\\
$N=10$%
\end{tabular}
\end{center}}}
{\parbox[b]{1.6302in}{\begin{center}
\includegraphics[
natheight=6.979000in,
natwidth=10.073300in,
height=1.132in,
width=1.6302in
]%
{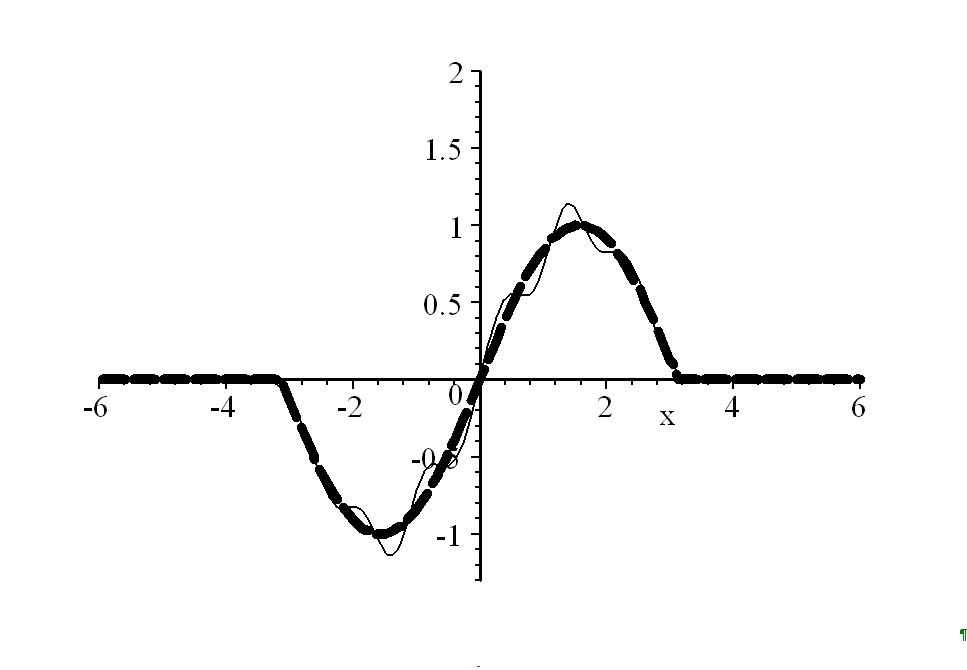}%
\\%
\begin{tabular}
[c]{c}%
Hermite expansion\\
$N=20$%
\end{tabular}
\end{center}}}
\]%
\[%
{\parbox[b]{1.6431in}{\begin{center}
\includegraphics[
natheight=6.698000in,
natwidth=9.801800in,
height=1.126in,
width=1.6431in
]%
{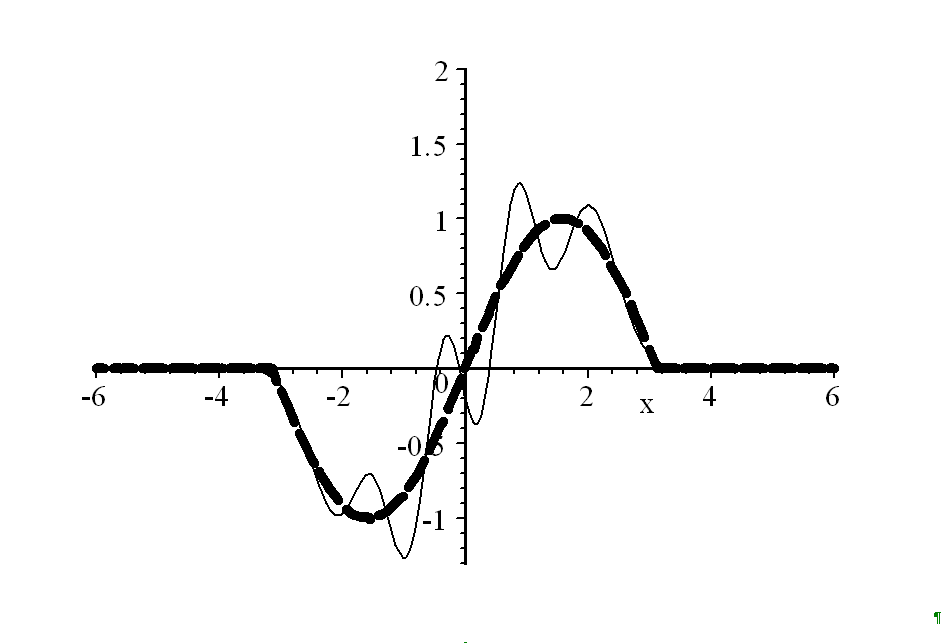}%
\\%
\begin{tabular}
[c]{l}%
Theorem \ref{AlgBump}\\
$N=10$, $t=.01$%
\end{tabular}
\end{center}}}
{\parbox[b]{1.6579in}{\begin{center}
\includegraphics[
natheight=6.906400in,
natwidth=10.114900in,
height=1.1346in,
width=1.6579in
]%
{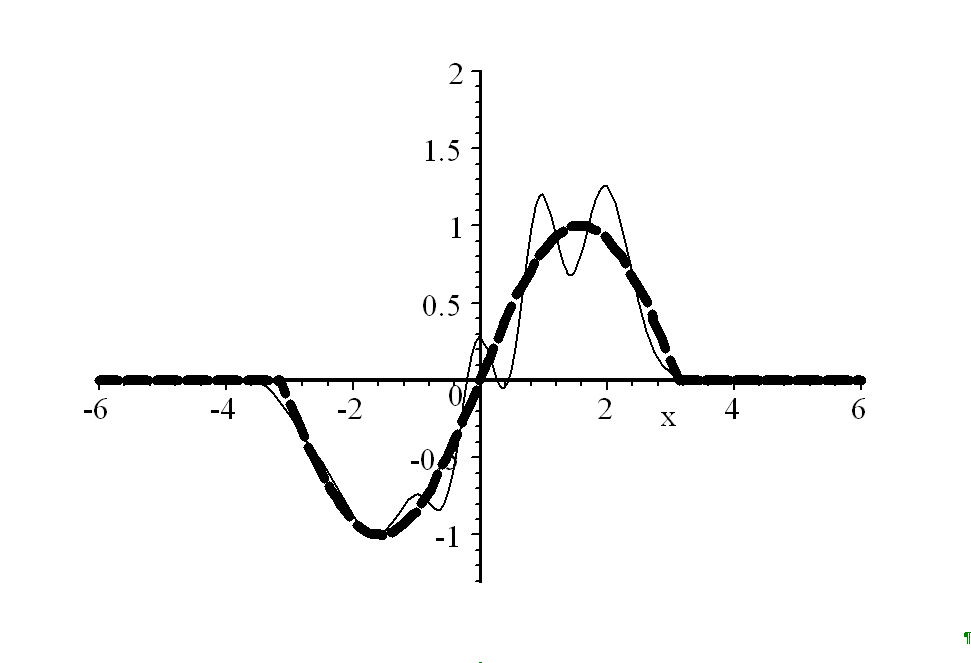}%
\\%
\begin{tabular}
[c]{l}%
Theorem \ref{AlgBump}\\
$N=20$, $t=.05$%
\end{tabular}
\end{center}}}
{\parbox[b]{1.6475in}{\begin{center}
\includegraphics[
natheight=6.906400in,
natwidth=10.114900in,
height=1.1277in,
width=1.6475in
]%
{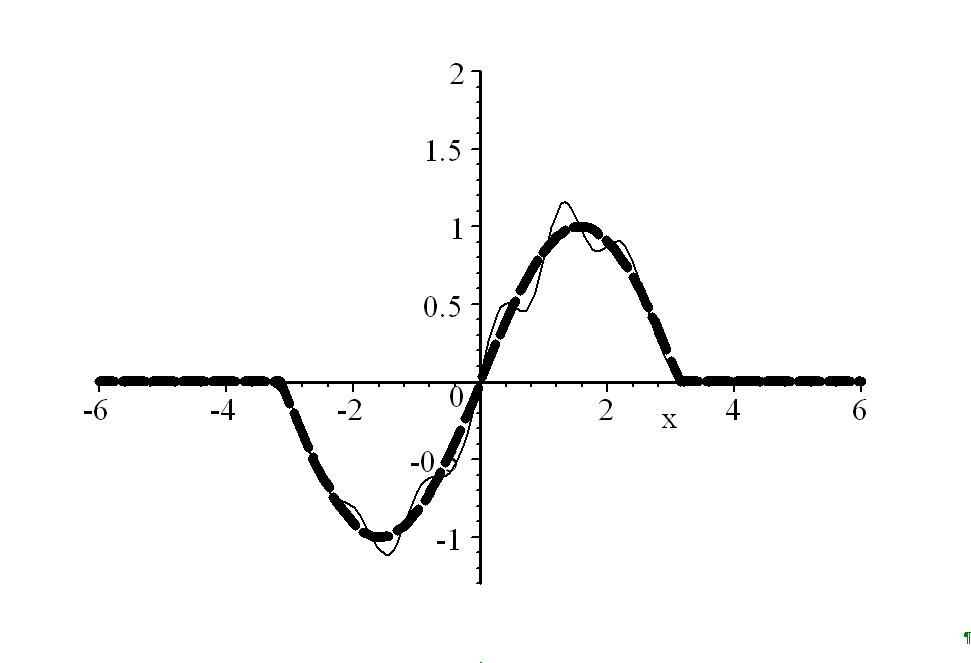}%
\\%
\begin{tabular}
[c]{l}%
Theorem \ref{AlgBump}\\
$N=20$, $t=.01$%
\end{tabular}
\end{center}}}
\]

\bigskip

In section \ref{SecImpConv} we review a standard technique accelerating this
convergence in $t$. In our experiments, though, we've found the Hermite
expansion is generally the bottleneck, not the round-off error of the
derivative approximations for $e^{-x^{2}}$.%
\[%
{\parbox[b]{1.4961in}{\begin{center}
\includegraphics[
natheight=6.770600in,
natwidth=9.916800in,
height=1.0222in,
width=1.4961in
]%
{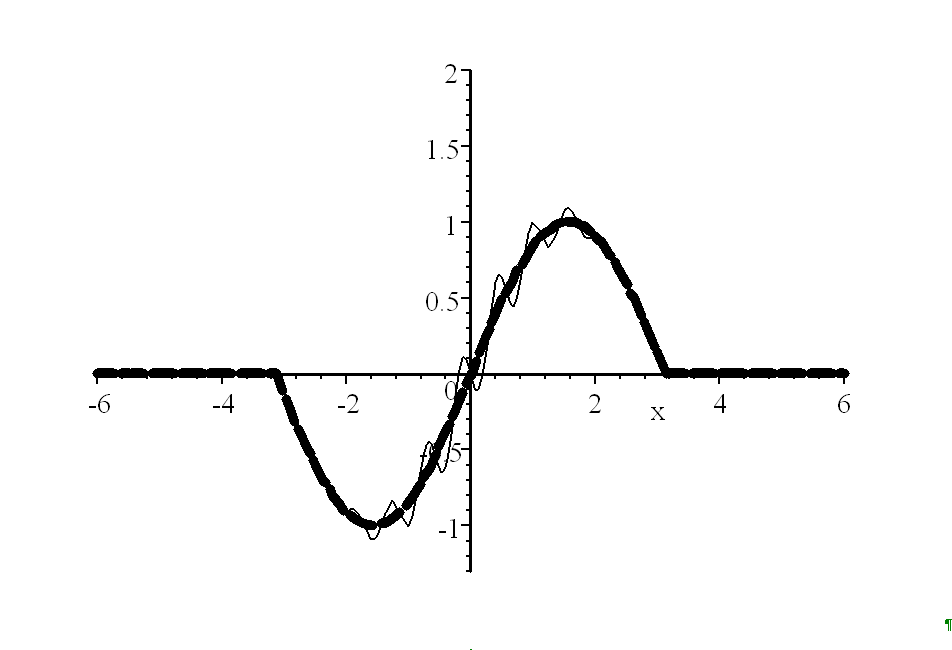}%
\\%
\begin{tabular}
[c]{c}%
Hermite expansion\\
$N=60$%
\end{tabular}
\end{center}}}
{\parbox[b]{1.4875in}{\begin{center}
\includegraphics[
natheight=6.854500in,
natwidth=10.031000in,
height=1.0188in,
width=1.4875in
]%
{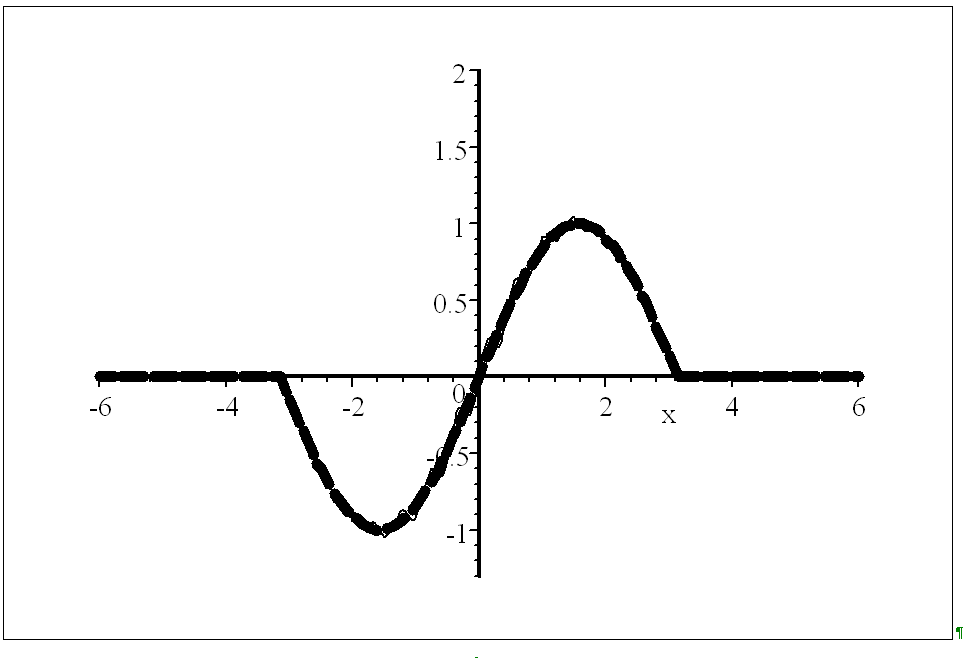}%
\\%
\begin{tabular}
[c]{c}%
Hermite expansion\\
$N=100$%
\end{tabular}
\end{center}}}
{\parbox[b]{1.5333in}{\begin{center}
\includegraphics[
natheight=6.687600in,
natwidth=10.135600in,
height=1.0144in,
width=1.5333in
]%
{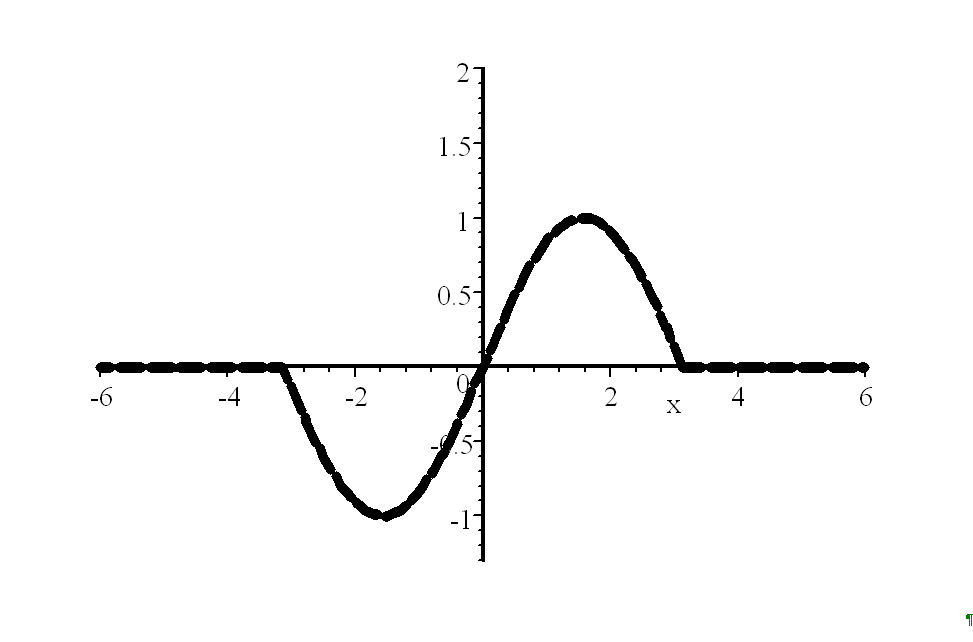}%
\\%
\begin{tabular}
[c]{c}%
Hermite expansion\\
$N=120$%
\end{tabular}
\end{center}}}
\]
We need about 120 terms before visual accuracy is achieved for this simple
function. There is a host of methods in the literature for improving
convergence of the Hermite expansion, but generally we have better success
with functions that are smooth and bounded \cite{Gottlieb}. Our last examples
in this section illustrate how convergence is faster for functions which are
smooth and ``clamped off'', meaning multiplied by $\left(  x-a\right)
^{n}\left(  x+a\right)  ^{n}\chi_{\left[  -a,a\right]  }$ whether or not they
are positive or symmetric.

\noindent%
{\parbox[b]{2.3281in}{\begin{center}
\includegraphics[
natheight=6.301900in,
natwidth=9.718800in,
height=1.5134in,
width=2.3281in
]%
{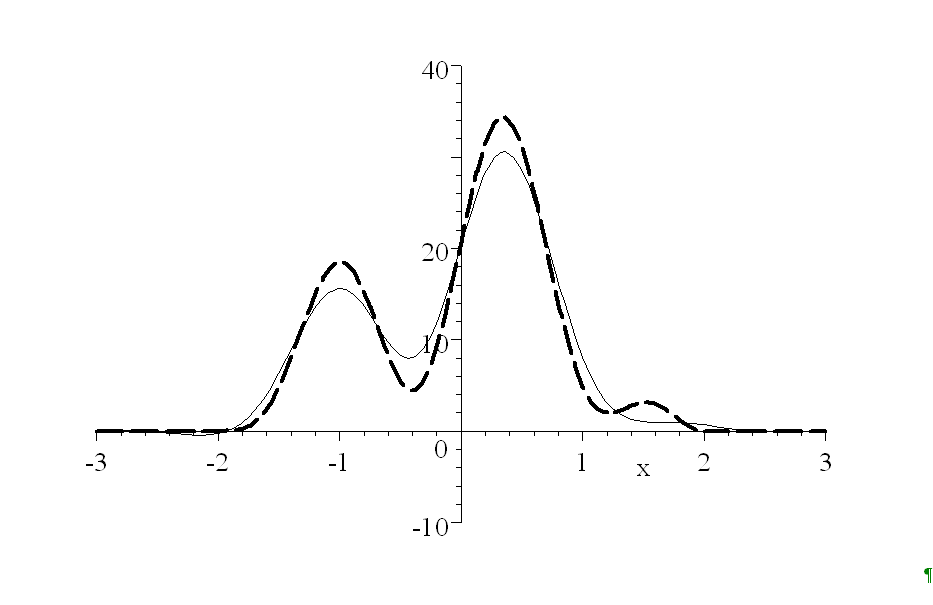}%
\\
Hermite $N=10$%
\end{center}}}
{\parbox[b]{2.3298in}{\begin{center}
\includegraphics[
natheight=6.718700in,
natwidth=10.364800in,
height=1.5152in,
width=2.3298in
]%
{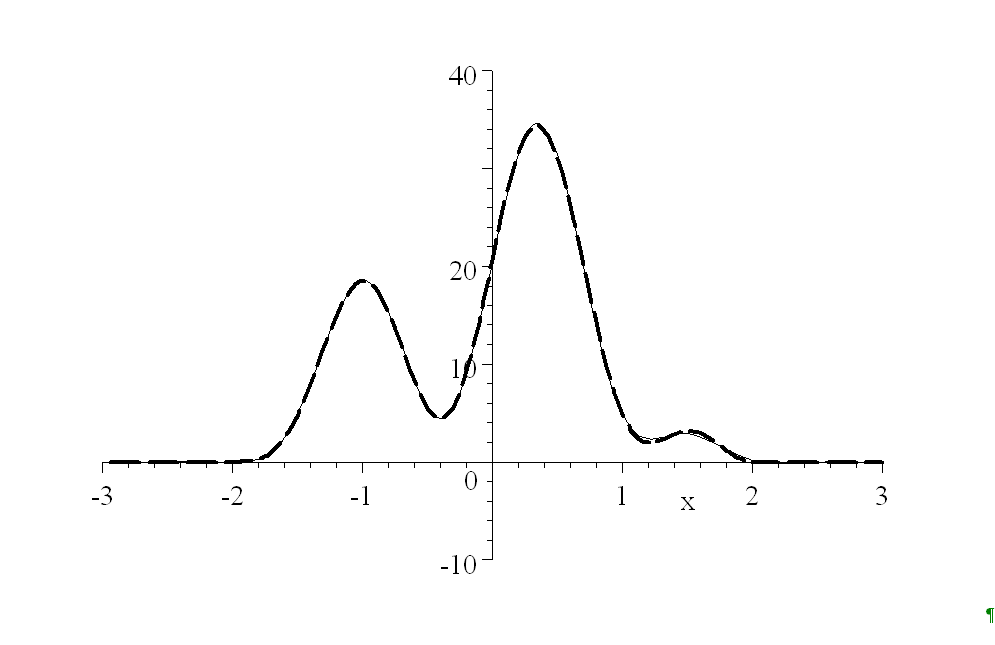}%
\\
Hermite $N=25$%
\end{center}}}

\noindent%
{\parbox[b]{2.335in}{\begin{center}
\includegraphics[
natheight=6.729100in,
natwidth=10.395900in,
height=1.516in,
width=2.335in
]%
{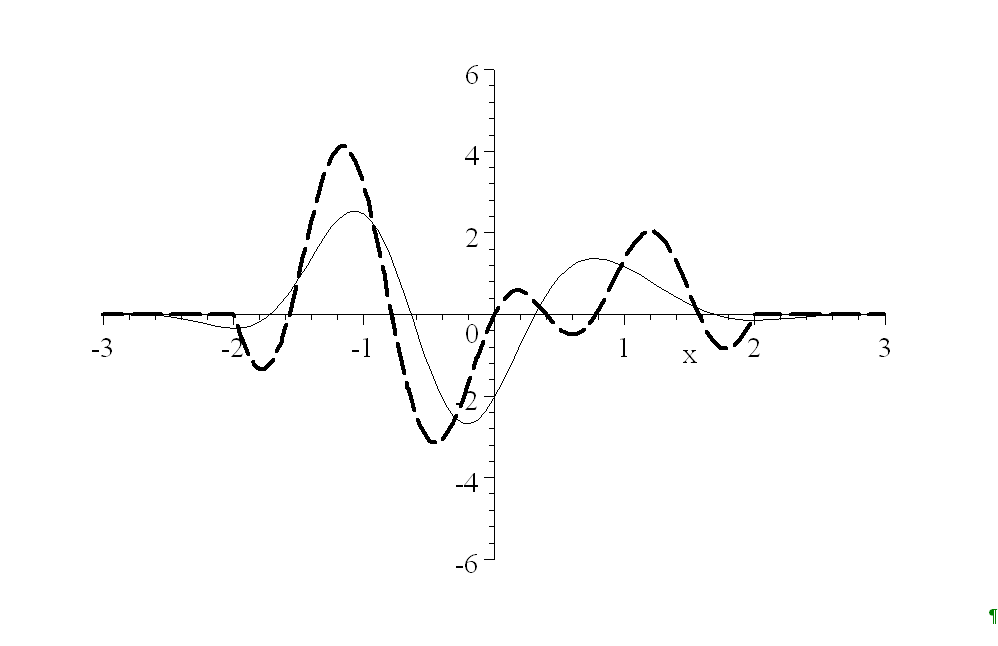}%
\\
Hermite $N=10$%
\end{center}}}
{\parbox[b]{2.3367in}{\begin{center}
\includegraphics[
natheight=6.698000in,
natwidth=10.333600in,
height=1.5195in,
width=2.3367in
]%
{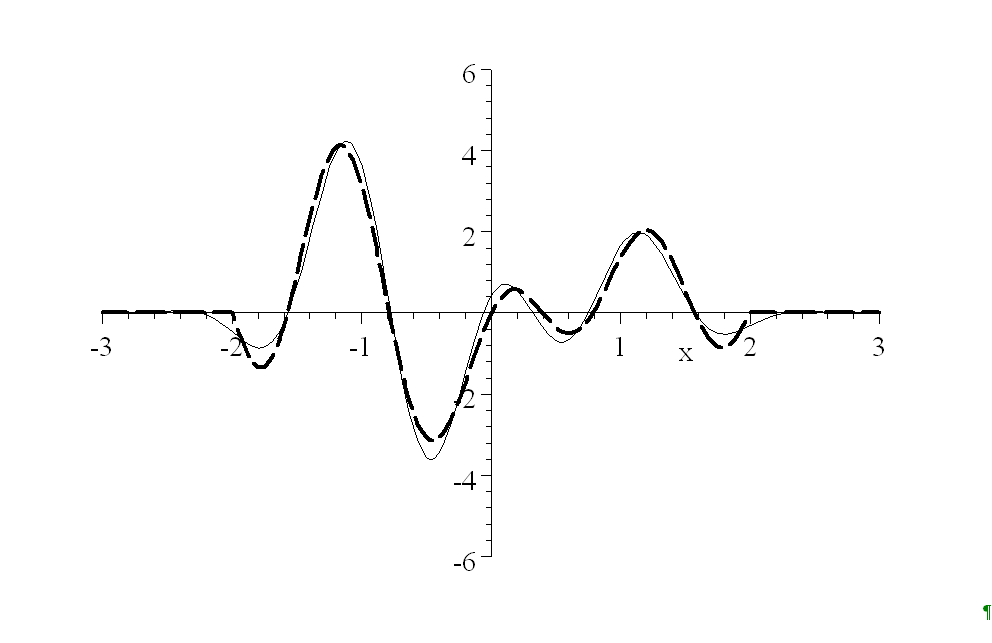}%
\\
Hermite $N=25$%
\end{center}}}

\bigskip

\section{Calculating the coefficients with least squares\label{SecLeastSqrs}}

Theorem \ref{ThmBumps} promises any $L^{2}$ function can be approximated
$f\left(  x\right)  \approx\overset{N}{\underset{n=0}{\sum}}a_{n}e^{-\left(
x-nt\right)  ^{2}}$. Theorem \ref{AlgBump} gives a formula for the
coefficients $a_{n}$ but this formula is not unique, and in fact is not
``best'' according to the classical continuous least squares technique.

\quad\quad%
{\parbox[b]{1.7521in}{\begin{center}
\includegraphics[
natheight=6.489500in,
natwidth=9.437700in,
height=1.2064in,
width=1.7521in
]%
{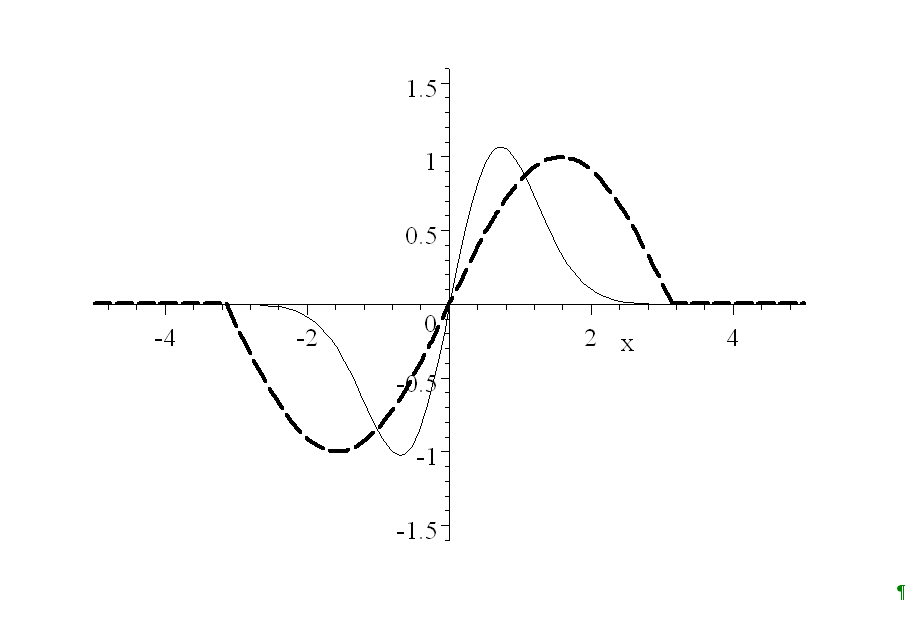}%
\\%
\begin{tabular}
[c]{l}%
Least squares approximation\\
$N=5$, $t=.01$%
\end{tabular}
\end{center}}}
\quad%
{\parbox[b]{1.7556in}{\begin{center}
\includegraphics[
natheight=6.604600in,
natwidth=9.625400in,
height=1.2064in,
width=1.7556in
]%
{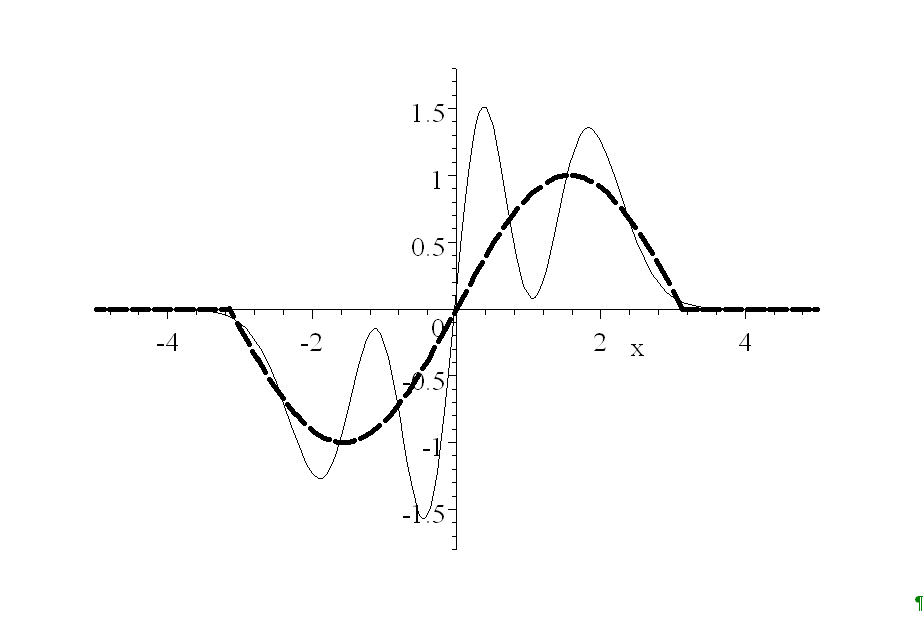}%
\\%
\begin{tabular}
[c]{l}%
Theorem \ref{AlgBump} approximation\\
$N=5$, $t=.01$%
\end{tabular}
\end{center}}}

\bigskip

\noindent In least squares we minimize the error function%
\[
E_{2}\left(  a_{0},...,a_{N}\right)  :=\underset{\mathbb{R}}{\int}\left|
f\left(  x\right)  -\overset{N}{\underset{n=0}{\sum}}a_{n}e^{-\left(
x-nt\right)  ^{2}}\right|  ^{2}dx
\]
by setting $\frac{\partial E_{2}}{\partial a_{j}}=0$ for $j=0,...,N$ and
solving for the $a_{n}$. These $N+1$ linear equations are called the
\textbf{normal equations}. The matrix form of this system is $M\overrightarrow
{v}=\overrightarrow{b}$ where $M$ is the matrix%
\[
M=\left[  \sqrt{\frac{\pi}{2}}e^{-\left(  k^{2}+j^{2}-\frac{\left(
k+j\right)  ^{2}}{2}\right)  t^{2}}\right]  _{j,k=0}^{N}%
\]
and%
\[
\overrightarrow{v}=\left[  a_{j}\right]  _{j=0}^{N}\text{\qquad and\qquad
}\overrightarrow{b}=\left[  \underset{\mathbb{R}}{\int}f\left(  x\right)
e^{-\left(  x-jt\right)  ^{2}}dx\right]  _{j=0}^{N}%
\]
$M$ is symmetric and invertible, so we can always solve for the $a_{n}$. But
these least squares matrices are notorious for being ill-conditioned when
using non-orthogonal approximating functions. The Hilbert matrix is the
archetypical example. The current application is no exception since the matrix
entries are very similar for most choices of $N$ and $t$,$\ $so round-off
error is extreme. Choosing $N=7$ instead of $5$ in the graphed example above
requires almost 300 significant digits.

\section{Low-frequency trig series are dense in $L^{2}$ with Gaussian
weight\label{SecLowFTrig}}

For $f\in L^{2}\left(  \mathbb{R},\mathbb{C}\right)  $ define the norm%
\[
\left\|  f\right\|  _{2,G}:=\left(  \underset{\mathbb{R}}{%
{\textstyle\int}
}\left|  f\left(  x\right)  \right|  ^{2}e^{-x^{2}}dx\right)  ^{1/2}\text{.}%
\]
Write $f\underset{\epsilon,G}{\approx}$ $g$ to mean $\left\|  f-g\right\|
_{2,G}<\epsilon$.

\begin{theorem}
\label{ThmLowFreq}For every $f\in L^{2}\left(  \mathbb{R},\mathbb{C}\right)  $
and $\epsilon>0$ there exists $N$ $\in\mathbb{N}$ and $t_{0}>0$ such that for
any $t\neq0$ with $\left|  t\right|  <t_{0}$
\[
f\left(  x\right)  \underset{\epsilon,G}{\approx}\text{ }\overset{N}%
{\underset{n=0}{%
{\textstyle\sum}
}}a_{n}e^{-intx}%
\]
for some choice of $a_{n}\in\mathbb{C}$ dependent on $N$ and $t$.
\end{theorem}

\begin{proof}
We use the Fourier transform with convention%
\[
\mathcal{F}\left[  f\right]  \left(  s\right)  =\frac{1}{\sqrt{2\pi}}%
\underset{\mathbb{R}}{%
{\textstyle\int}
}f\left(  x\right)  e^{-isx}dx\text{.}%
\]
$\mathcal{F}$ is a linear isometry of $L^{2}\left(  \mathbb{R},\mathbb{C}%
\right)  $ with%
\begin{align*}
\mathcal{F}\left[  e^{-\alpha x^{2}}\right]   & =\frac{1}{\sqrt{2\alpha}%
}e^{-\frac{s^{2}}{4\alpha}}\text{,}\\
\mathcal{F}\left[  f\left(  x+r\right)  \right]   & =e^{-irs}\mathcal{F}%
\left[  f\left(  x\right)  \right]  \text{\qquad and}\\
\mathcal{F}\left[  g\ast h\right]   & =\sqrt{2\pi}\mathcal{F}\left[  g\right]
\mathcal{F}\left[  h\right]  \text{.}%
\end{align*}
where $\ast$ is convolution.

Let $f\in L^{2}$ and we now show $f_{2}\left(  x\right)  :=\frac{1}{\sqrt
{2\pi}}e^{-x^{2}}\ast\mathcal{F}^{-1}\left[  f\right]  \left(  x\right)  \in
L^{2}$. Notice $g:=\mathcal{F}^{-1}\left[  f\right]  \in L^{2}$ and%
\begin{align*}
\left\|  f_{2}\right\|  _{2}^{2}  & =\underset{\mathbb{R}}{%
{\textstyle\int}
}\left|  \underset{\mathbb{R}}{%
{\textstyle\int}
}\frac{1}{\sqrt{2\pi}}g\left(  x-y\right)  e^{-y^{2}}dy\right|  ^{2}%
ds\leq\frac{1}{2\pi}\underset{\mathbb{R}}{%
{\textstyle\int}
}\underset{\mathbb{R}}{%
{\textstyle\int}
}\left|  g\left(  x-y\right)  \right|  ^{2}e^{-2y^{2}}dyds\\
& =c\left\|  \mathcal{W}_{t_{0}}\left[  \left|  g\right|  ^{2}\right]
\right\|  _{1}=c\left\|  g^{2}\right\|  _{1}=c\left\|  g\right\|  _{2}%
^{2}=c\left\|  f\right\|  _{2}^{2}<\infty
\end{align*}
for some $c>0$. Here $\mathcal{W}_{t}\left[  h\right]  $ is the solution to
the diffusion equation for time $t$ and initial condition $h$. (The notation
$\mathcal{W}$ refers to the Weierstrass transform.) The reason for the third
equality in the previous calculation is that $\mathcal{W}_{t}$ maintains the
$L^{1}$ integral of any positive initial condition $h$ for all time $t>0$
\cite{WidderHeat}.

Now approximate the real and imaginary parts of $f_{2}$ with Theorem
\ref{AlgBump}. Then we get%
\[
\tfrac{1}{\sqrt{2\pi}}e^{-x^{2}}\ast\mathcal{F}^{-1}\left[  f\right]  \left(
x\right)  \underset{\epsilon}{\approx}\ \overset{N}{\underset{n=0}{%
{\textstyle\sum}
}}a_{n}e^{-\left(  x-nt\right)  ^{2}}\text{\qquad}a_{n}\in\mathbb{C}%
\]
and applying $\mathcal{F}$ gives%
\[
\tfrac{1}{\sqrt{2}}e^{-s^{2}/4}f\left(  s\right)  \underset{\epsilon}{\approx
}\text{ }\overset{N}{\underset{n=0}{%
{\textstyle\sum}
}}a_{n}e^{-ints}\tfrac{1}{\sqrt{2}}e^{-s^{2}/4}%
\]
Hence%
\[
f\left(  s\right)  \underset{\sqrt{2}\epsilon,G}{\approx}\text{ }\overset
{N}{\underset{n=0}{%
{\textstyle\sum}
}}a_{n}e^{-ints}%
\]
using the fact that $e^{-s^{2}/4}>e^{-s^{2}}$.
\end{proof}

This result is surprising, even in the context of this paper, because for
instance, series of the form $\overset{N}{\underset{n=-N}{%
{\textstyle\sum}
}}a_{n}e^{-i\left(  x+nt\right)  }$ for all $t$ and $a_{n}$ are not dense in
$L^{2}$ and in fact only inhabit a 4-dimensional subspace of the infinite
dimensional Hilbert space \cite{CalcFoliation}.

\begin{corollary}
On any finite interval $\left[  a,b\right]  $ for any $\omega>0$ the finite
linear combinations of sine and cosine functions with frequency lower than
$\omega$ are dense in $L^{2}\left(  \left[  a,b\right]  ,\mathbb{R}\right)  $.
\end{corollary}

\begin{proof}
On $\left[  a,b\right]  $ the Gaussian is bounded and so the norms with or
without weight function are equivalent. Apply Theorem \ref{ThmLowFreq} to
$f\in L^{2}\left(  \left[  a,b\right]  ,\mathbb{R}\right)  $ and choose $t$
such that $Nt<\omega$ to get%
\[
f\underset{\epsilon}{\approx}\text{ }\overset{N}{\underset{n=0}{%
{\textstyle\sum}
}}\operatorname{Re}\left(  a_{n}\right)  \cos\left(  ntx\right)
+\operatorname{Im}\left(  a_{n}\right)  \sin\left(  ntx\right)
\]
where
\[
a_{n}=\frac{\left(  -1\right)  ^{n}}{n!2\pi}\overset{N}{\underset{k=n}{%
{\textstyle\sum}
}}\tfrac{1}{\left(  k-n\right)  !\left(  2t\right)  ^{k}}\underset{\mathbb{R}%
}{%
{\textstyle\int}
}\left[  e^{-x^{2}}\ast\mathcal{F}^{-1}\left[  f\right]  \left(  x\right)
\right]  e^{x^{2}}\frac{d^{k}}{dx^{k}}\left(  e^{-x^{2}}\right)  dx\text{.}%
\]
\end{proof}

Applying Remark \ref{RemDense} to this result shows even discrete sets of
positive frequencies that approach 0 make the span of the corresponding sine
and cosine functions equal to$L^{2}\left(  \left[  a,b\right]  ,\mathbb{R}%
\right)  $.

Finally, low-frequency cosines span the even functions:

\begin{proposition}
On any finite interval $\left[  0,b\right]  $ for any $\omega>0$ the finite
linear combinations of cosine functions with frequency lower than $\omega$ are
dense in $L^{2}\left(  \left[  0,b\right]  ,\mathbb{R}\right)  $.
\end{proposition}

\begin{proof}
Let $f\in L^{2}\left(  \left[  0,b\right]  ,\mathbb{R}\right)  $ and extend it
as an even function on $\left[  -b,b\right]  $. Now use the previous corollary
to write
\[
f\underset{\epsilon}{\approx}\text{ }\overset{N}{\underset{n=0}{%
{\textstyle\sum}
}}a_{n}\cos\left(  ntx\right)  +b_{n}\sin\left(  ntx\right)  \text{.}%
\]
We'd like to conclude right now that the $b_{n}=0$ or $b_{n}\approx0$, but
that is not true. However, every function $g$ on $\left[  -b,b\right]  $ may
be written uniquely as a sum of even and odd functions%
\begin{align*}
g  & =g_{e}+g_{o}\\
g_{e}\left(  x\right)   & =\frac{g\left(  x\right)  +g\left(  -x\right)  }%
{2}\\
g_{e}\left(  x\right)   & =\frac{g\left(  x\right)  -g\left(  -x\right)  }{2}%
\end{align*}
and so%
\[
g\underset{\epsilon}{\approx}\text{ }h\text{\quad}\Rightarrow\text{\quad}%
g_{e}\underset{\epsilon}{\approx}\text{ }h_{e}\text{.}%
\]
Therefore%
\[
f=f_{e}\underset{\epsilon}{\approx}\text{ }\left[  \overset{N}{\underset{n=0}{%
{\textstyle\sum}
}}a_{n}\cos\left(  ntx\right)  +b_{n}\sin\left(  ntx\right)  \right]
_{e}=\overset{N}{\underset{n=0}{%
{\textstyle\sum}
}}a_{n}\cos\left(  ntx\right)  \text{.}%
\]
\end{proof}

Beware this last result; it's not as strong as Fourier approximation. The
coefficients for the sine functions calculated above may be large; the
proposition merely promises the linear combination of the sine terms is small.
Using least squares, however, will have vanishing sine coefficients.

\section{Origins and generalizations}

The mathematical inspiration for Theorem \ref{ThmBumps} comes from geometrical
investigations in infinite dimensional control theory. We noticed that
function translation and vector translation in $L^{2}\left(  \mathbb{R}%
\right)  $ do not commute. Specifically, ``function translation'' is a flow on
the infinite dimensional vector space $L^{2}\left(  \mathbb{R}\right)  $ given
by the map $F:L^{2}\left(  \mathbb{R}\right)  \times\mathbb{R}\rightarrow
L^{2}\left(  \mathbb{R}\right)  $ where $F_{t}\left(  f\right)  \left(
x\right)  :=f\left(  x+t\right)  $. ``Vector translation'' in the direction of
$g\in L^{2}\left(  \mathbb{R}\right)  $ is the flow $G:L^{2}\left(
\mathbb{R}\right)  \times\mathbb{R}\rightarrow L^{2}\left(  \mathbb{R}\right)
$ where $G_{t}\left(  f\right)  :=f+tg$. Taking for example $g\left(
x\right)  :=e^{-x^{2}}$ and composing $F$ and $G$ we see $F_{t}\circ G_{t}\neq
G_{t}\circ F_{t}$ since for $f\equiv0$%
\[
F_{t}\circ G_{t}\left(  f\right)  \left(  x\right)  =te^{-\left(  x+t\right)
^{2}}\text{ \qquad while\qquad}G_{t}\circ F_{t}\left(  f\right)  \left(
x\right)  =te^{-x^{2}}\text{.}%
\]
Notice however the key fact%
\[
\frac{F_{t}\circ G_{t}-G_{t}\circ F_{t}}{t^{2}}\left(  f\right)
\rightarrow\frac{d}{dx}\left(  e^{-x^{2}}\right)  \text{\qquad as
}t\rightarrow0
\]
In finite dimensions the commutator quotient above gives the Lie bracket
$\left[  X,Y\right]  $ of the vector fields $X$ and $Y$ which generate the
flows $F$ and $G$, respectively. A fundamental result in finite-dimensional
control theory states that the reachable set via $X$ and $Y$ is given by the
integral surface to the distribution made up of iterated Lie brackets starting
from $X$ and $Y$ (Chow's Theorem, which is an interpretation of Frobenius'
Foliation Theorem, see \cite{Sontag}, e.g.). The idea we are exploiting is
that iterated Lie brackets for our flows $F$ and $G$ will give successive
derivatives of the Gaussian, whose span is dense in $L^{2}\left(
\mathbb{R}\right)  $. Consequently, the reachable set via $F$ and $G$ from
$f\equiv0$ should be all of $L^{2}\left(  \mathbb{R}\right)  $. That is to
say, sums of translates and multiples of one Gaussian (with fixed variance)
can approximate any integrable function.

Unfortunately this program doesn't automatically work on the infinite
dimensional vector space $L^{2}\left(  \mathbb{R}\right)  $ since the function
translation flow is not generated by a simple vector field on $L^{2}\left(
\mathbb{R}\right)  $. So instead of studying vector fields, we consider flows
as primary. The fundamental results can be rewritten and still hold in the
general context of a metric space \cite{CalcFoliation}. Then other functions
besides $g\left(  x\right)  =e^{-x^{2}}$ can be checked to be derivative
generating and other flows may be used in place of translation. E.g., Fourier
approximation is achieved using dilation $F:L^{2}\left(  \mathbb{R}%
,\mathbb{C}\right)  \times\mathbb{R}\rightarrow L^{2}\left(  \mathbb{R}%
,\mathbb{C}\right)  $ where $F_{t}\left(  f\right)  \left(  x\right)
:=f\left(  e^{t}x\right)  $ and $G_{t}\left(  f\right)  \left(  x\right)
:=f\left(  x\right)  +te^{ix}$. This gives us a general tool for determining
the density of various families of functions.

Another opportunity for generalizing the results of this paper presents itself
with the observation that Hermite expansions are valid for functions defined
on $\mathbb{C}$ or $\mathbb{R}^{n}$ and in spaces of tempered distributions;
and divided differences works in all of these spaces as well.

Note also that while the results of section \ref{SecAppAlgo} work for uniform
approximations of continuous functions on finite intervals (Remark
\ref{RemUnifG}), this is an open question for low-frequency trigonometric approximations.

The results of this paper can be ported to the language of control theory
where we can then conclude the system%
\begin{equation}
u_{t}=c_{1}\left(  t\right)  u_{x}+c_{2}(t)e^{-x^{2}}\label{LineControl2}%
\end{equation}
is bang-bang controllable with controls of the form $c_{1},c_{2}%
:\mathbb{R}^{+}\rightarrow\left\{  -1,0,1\right\}  $. Theorem \ref{AlgBump}
drives the initial condition $f\equiv0$ to any state in $L^{2}$ under the
system $\left(  \ref{LineControl2}\right)  $, but may be nowhere near optimal
for approximating a function such as $e^{-\left(  x+10\right)  ^{2}}$, since
it uses only Gaussians $e^{-\left(  x+s\right)  ^{2}}$ with choices of $s<<10$.

Finally, interpreting Theorem \ref{ThmBumps} in terms of signal analysis, we
see a Gaussian filter is a universal synthesizer with arbitrarily short load
time. Let $G\left(  x\right)  :=\frac{1}{\sqrt{\pi}}e^{-x^{2}}$. A Gaussian
filter is a linear time-invariant system represented by the operator%
\[
\mathcal{W}\left(  f\right)  \left(  x\right)  :=\left(  f\ast G\right)
\left(  x\right)  =\frac{1}{\sqrt{\pi}}\int_{\mathbb{R}}f\left(  y\right)
e^{-\left(  s-x\right)  ^{2}}dy\text{.}%
\]
Notice if you feed $\mathcal{W}$ a Dirac delta distribution $\delta_{t}$ (an
ideal impulse at time $x=t$) you get $\mathcal{W}\left(  \delta_{t}\right)
=G\left(  x-t\right)  $. Then Theorem \ref{ThmBumps} gives

\begin{corollary}
For any $f\in L^{2}\left(  \mathbb{R}\right)  $ and any $\epsilon>0$ and any
$\tau>0$ there exists $t>0$ and $N\in\mathbb{N}$ with $tN<\tau$ such that%
\[
f\underset{\epsilon}{\approx}\mathcal{W}\left(  \overset{N}{\underset{n=0}{%
{\textstyle\sum}
}}a_{n}\delta_{nt}\right)
\]
for some choice of $a_{n}\in\mathbb{R}$.
\end{corollary}

Feed a Gaussian filter a linear combination of impulses and we can synthesize
any signal and arbitrarily small load time $\tau$. The design of physical
approximations to an analog Gaussian filter are detailed in \cite{Dishal},
\cite{Madrenas}.

\section{Appendix\label{SecImpConv}: Approximating higher derivatives}

The results in this paper may be much improved with voluminous techniques
available from numerical analysis. E.g., \cite{Greengard} gives an algorithm
which speeds the calculation of sums of Gaussians, and \cite{Leibon} explores
Hermite expansion acceleration useful in step 1 of the proof of Theorem
\ref{ThmBumps}. This section is devoted to reviewing methods which improve the
error in step 2, approximating derivatives of the Gaussian with finite
differences. We also derive the error formula used in Proposition
\ref{propLpDivDiff}.

Above we approximated derivatives with the formula%
\begin{equation}
\frac{d^{n}}{dx^{n}}f\left(  x\right)  =%
\begin{tabular}
[c]{c}%
$\underbrace{\frac{1}{t^{n}}%
{\textstyle\sum_{k=0}^{n}}
\left(  -1\right)  ^{n-k}\binom{n}{k}f\left(  x+kt\right)  }$\\
gives round-off error as $t\rightarrow0^{+}$%
\end{tabular}%
\begin{tabular}
[c]{c}%
$\underset{}{+}$%
\end{tabular}%
\begin{tabular}
[c]{l}%
$\underbrace{O\left(  t\right)  }$\\
truncation error
\end{tabular}
\text{.}\label{LineNthDer=O(t)}%
\end{equation}
The N\"{o}rlund-Rice integral may be of interest for extremely large $n$ as it
avoids the calculation of the binomial coefficient by evaluating a complex
integral. In this section, though, we devote our attention to deriving
$n$-point formulas; these formulas decrease round-off error by increasing the
number of evaluations $f\left(  x+kt\right)  $--this shrinks the truncation
error without sending $t\rightarrow0$.

In approximating the $k$th derivative with an $n+1$ point formula%
\[
f^{\left(  k\right)  }\left(  x\right)  \approx\frac{1}{t^{k}}\overset
{n}{\underset{i=0}{%
{\textstyle\sum}
}}c_{i}f\left(  x+k_{i}t\right)
\]
we wish to calculate the coefficients $c_{i}$. In the forward difference
method, the $k_{i}=i$, but keeping these values general allows us to find the
coefficients for the central or backward difference formulas just as easily.
The following method for finding the $c_{i}$ was shown to us by our student
Jeffrey Thornton who rediscovered the formula.

Taylor's Theorem has%
\[
f\left(  x+k_{i}t\right)  =\overset{n}{\underset{j=0}{%
{\textstyle\sum}
}}\frac{\left(  k_{i}t\right)  ^{j}}{j!}f^{\left(  j\right)  }\left(
x\right)  +\frac{\left(  k_{i}t\right)  ^{n+1}}{\left(  n+1\right)
!}f^{\left(  n+1\right)  }\left(  \xi_{i}\right)
\]
for some $\xi_{i}$ between $x$ and $x+k_{i}t$. From this it follows%
\begin{align*}
& \overset{n}{\underset{i=0}{%
{\textstyle\sum}
}}c_{i}f\left(  x+k_{i}t\right) \\
& =\left[
\begin{tabular}
[c]{c}%
$f\left(  x\right)  $\\
$tf^{\prime}\left(  x\right)  $\\
$\vdots$\\
$t^{n}f^{\left(  n\right)  }\left(  x\right)  $\\
$t^{n+1}$%
\end{tabular}
\right]  ^{T}\left[
\begin{tabular}
[c]{cccc}%
$1$ & $1$ & $\cdots$ & $1$\\
$k_{0}$ & $k_{1}$ & $\cdots$ & $k_{n}$\\
$\frac{k_{0}^{2}}{2!}$ & $\frac{k_{1}^{2}}{2!}$ & $\cdots$ & $\frac{k_{n}^{2}%
}{2!}$\\
$\vdots$ & $\vdots$ & $\ddots$ & $\vdots$\\
$\frac{k_{0}^{n}}{n!}$ & $\frac{k_{1}^{n}}{n!}$ & $\cdots$ & $\frac{k_{n}^{n}%
}{n!}$\\
$\tfrac{k_{0}^{n+1}f^{\left(  n+1\right)  }\left(  \xi_{0}\right)  }{\left(
n+1\right)  !}$ & $\frac{k_{1}^{n+1}f^{\left(  n+1\right)  }\left(  \xi
_{1}\right)  }{\left(  n+1\right)  !}$ & $\cdots$ & $\frac{k_{n}%
^{n+1}f^{\left(  n+1\right)  }\left(  \xi_{n}\right)  }{\left(  n+1\right)
!}$%
\end{tabular}
\right]  \left[
\begin{tabular}
[c]{c}%
$c_{0}$\\
$c_{1}$\\
$\vdots$\\
$c_{n}$%
\end{tabular}
\right]
\end{align*}
Now pick $c=\left[  c_{i}\right]  $ as a solution to%
\begin{equation}
\left[
\begin{tabular}
[c]{cccc}%
$1$ & $1$ & $\cdots$ & $1$\\
$k_{0}$ & $k_{1}$ & $\cdots$ & $k_{n}$\\
$\frac{k_{0}^{2}}{2!}$ & $\frac{k_{1}^{2}}{2!}$ & $\cdots$ & $\frac{k_{n}^{2}%
}{2!}$\\
$\vdots$ & $\vdots$ & $\ddots$ & $\vdots$\\
$\frac{k_{0}^{n}}{n!}$ & $\frac{k_{1}^{n}}{n!}$ & $\cdots$ & $\frac{k_{n}^{n}%
}{n!}$%
\end{tabular}
\right]  \left[
\begin{tabular}
[c]{c}%
$c_{0}$\\
$c_{1}$\\
$\vdots$\\
$c_{n}$%
\end{tabular}
\right]  =\left[
\begin{tabular}
[c]{c}%
$0$\\
$\vdots$\\
$1$\\
$\vdots$\\
$0$%
\end{tabular}
\right] \label{LineNumDiffCoeffMatrix}%
\end{equation}
which is possible since the $k_{i}$ are different, so the matrix is
invertible, as is seen using the Vandermonde determinant%
\[
\det=\frac{\underset{0\leq i<j\leq n}{\Pi}\left(  k_{j}-k_{i}\right)
}{\underset{2\leq i\leq n}{\Pi}i!}\text{.}%
\]
Then we must have%
\begin{align*}
\overset{n}{\underset{i=0}{%
{\textstyle\sum}
}}c_{i}f\left(  x+k_{i}t\right)   & =\left[
\begin{tabular}
[c]{c}%
$f\left(  x\right)  $\\
$tf^{\prime}\left(  x\right)  $\\
$\vdots$\\
$t^{n}f^{\left(  n\right)  }\left(  x\right)  $\\
$t^{n+1}$%
\end{tabular}
\right]  ^{T}\left[
\begin{tabular}
[c]{l}%
$0$\\
$\vdots$\\
$1$\quad($k$-th position)\\
$\vdots$\\
$0$\\
$\frac{1}{\left(  n+1\right)  !}\overset{n}{\underset{i=0}{%
{\textstyle\sum}
}}c_{i}k_{i}^{n+1}f^{\left(  n+1\right)  }\left(  \xi_{i}\right)  $%
\end{tabular}
\right] \\
& =t^{k}f^{\left(  k\right)  }\left(  x\right)  +\frac{t^{n+1}}{\left(
n+1\right)  !}\overset{n}{\underset{i=1}{%
{\textstyle\sum}
}}c_{i}k_{i}^{n+1}f^{\left(  n+1\right)  }\left(  \xi_{i}\right)  \text{.}%
\end{align*}
Therefore%
\[
f^{\left(  k\right)  }\left(  x\right)  =\frac{1}{t^{k}}\overset{n}%
{\underset{i=0}{%
{\textstyle\sum}
}}c_{i}f\left(  x+k_{i}t\right)  +Error
\]
for $c_{i}$ which satisfy $\left(  \ref{LineNumDiffCoeffMatrix}\right)  $
where%
\[
Error=-\dfrac{t^{n+1-k}}{\left(  n+1\right)  !}\overset{n}{\underset{i=0}{%
{\textstyle\sum}
}}c_{i}k_{i}^{n+1}f^{\left(  n+1\right)  }\left(  \xi_{i}\right)  \text{.}%
\]
This $Error$ formula shows how truncation error may be decreased by increasing
$n$ without shrinking $t$, thus combatting round-off error at the expense of
increased computation of sums.

The coefficients in $\left(  \ref{LineNthDer=O(t)}\right)  $ are obtained by
solving $M$ for the $c_{i}$ with $k_{i}$ chosen as $k_{i}=i$.

Thornton also points out that the $k_{i}$ may be chosen as complex values when
$f$ is analytic (as is the case with our Gaussians). This gives us another
opportunity to mitigate round-off error, since a greater quantity of
regularly-spaced nodes $k_{i}$ can be packed into an epsilon ball around zero
in the complex plane than on the real line.

As final note we mention there have been numerous advances to the present day
in inverting the Vandermonde matrix. We mention only the earliest application
to numerical differentiation \cite{Spitzbart} which gives a formula in terms
of the Stirling numbers.


\begin{thebibliography}{9}                                                                                                %
\bibitem {Bensoussan}Alain Bensoussan, et al., ``Representation and Control of
Infinite Dimensional Systems,'' 2nd ed., Springer, 2006.

\bibitem {Bilodeau}G. G. Bilodeau, The Weierstrass Transform and Hermite
Polynomials, Duke Mathematical Journal, Vol. 29, No. 2, 1962.

\bibitem {CalcFoliation}Craig Calcaterra, Foliating Metric Spaces, preprint,
arXiv:math/0608416, 2006.

\bibitem {CalcGaussians}Craig Calcaterra, Linear Combinations of Gaussians
with a Single Variance are dense in $L^{2}$, Proceedings of the World Congress
on Engineering, 2008.

\bibitem {Darlington}S. Darlington, Synthesis and Reactance of 4-poles, J.
Math. \& Phys., 18, pp. 257-353, 1939.

\bibitem {Dishal}Milton Dishal, Gaussian-Response Filter Design, Electrical
Communication, Volume 36, No. 1, pp. 3-26, 1959.

\bibitem {Gottlieb}David Gottlieb and Steven Orszag, ``Numerical Analysis of
Spectral Methods,'' SIAM, 1977.

\bibitem {Greengard}Leslie Greengard and Xiaobai Sun, A New Version of the
Fast Gauss Transform, Documenta Mathematica, Extra Volume ICM 1998, III, pp. 575-584.

\bibitem {Knuth}Donald Knuth, ``Concrete Mathematics,'' 2nd ed.,
Addison-Wesley, 1994.

\bibitem {Leibon}Greg Leibon, Daniel Rockmore \& Gregory Chirikjian, A Fast
Hermite Transform with Applications to Protein Structure Determination,
Proceedings of the 2007 international Workshop on Symbolic-Numeric
Computation, ACM, New York, NY, pp. 117-124, 2007.

\bibitem {Madrenas}J. Madrenas, M. Verleysen, P. Thissen, and J. L. Voz, A
CMOS Analog Circuit for Gaussian Functions, IEEE Transactions on Circuits and
Systems-II: Analog and Digital Signal Processing, Vol. 43, No. 1, 1996.

\bibitem {Ralston}Anthony Ralston and Philip Rabinowitz, ``A First Course in
Numerical Analysis,'' McGraw-Hill, 1978.

\bibitem {Sontag}Eduardo D. Sontag, ``Mathematical Control Theory,'' 2nd Ed.,
Springer-Verlag, 1998.

\bibitem {Spitzbart}A. Spitzbart and N. Macon, Numerical Differentiation
Formulas, The American Mathematical Monthly, Vol. 64, No. 10, pp. 721-723, 1957.

\bibitem {Stone}M. H. Stone, Developments in Hermite Polynomials, The Annals
of Mathematics, 2nd Ser., Vol. 29, No. 1/4, pp. 1-13, 1927-1928.

\bibitem {Szego}Gabor Szeg\"{o}, ``Orthogonal Polynomials,'' American
Mathematical Society, 3rd ed., 1967.

\bibitem {WidderHeat}David Widder, ``The Heat Equation,'' Pure and Applied
Mathematics, Vol. 67. Academic Press, 1975.
\end{thebibliography}
\end{document}